\documentclass[12pt, reqno]{amsart}
\usepackage{amsmath, amsthm, amscd, amsfonts, amssymb, graphicx, color}
\usepackage[bookmarksnumbered, colorlinks, plainpages]{hyperref}
\hypersetup{colorlinks=true,linkcolor=red, anchorcolor=green, citecolor=cyan, urlcolor=red, filecolor=magenta, pdftoolbar=true}

\newtheorem{theorem}{Theorem}[section]
\newtheorem{lemma}[theorem]{Lemma}
\newtheorem{proposition}[theorem]{Proposition}
\newtheorem{corollary}[theorem]{Corollary}
\theoremstyle{definition}
\newtheorem{definition}[theorem]{Definition}

\theoremstyle{remark}
\newtheorem{remark}[theorem]{Remark}
\numberwithin{equation}{section}
\DeclareMathOperator{\cone}{cone}
\DeclareMathOperator{\face}{face}
\DeclareMathOperator{\co}{co}

\begin{document}

\setcounter{page}{1}

\title[$M$-ideals and split faces]{$M$-ideals and split faces of the quasi state space of a  non-unital ordered Banach space}

\author{ Anindya Ghatak$^1$ \MakeLowercase{and} Anil Kumar Karn$^1$}

\address{$^1$School of Mathematical Sciences, 
	National Institute of Science Education and Research, HBNI, Bhubaneswar, 
	At and P.O.- Jatni, Dist- Khurda, PIN-752050, India.} 

\email{\textcolor[rgb]{0.00,0.00,0.84}{anindya.ghatak@niser.ac.in; anilkarn@niser.ac.in}}

%\dedicatory{This paper is dedicated to Professor ABCD}



\date{}

\begin{abstract}
	 We characterize $M$-ideals in order smooth $\infty$-normed spaces by extending the notion of split faces of the state space to those of the quasi-state space. We also characterize approximate order unit spaces as those order smooth $\infty$-normed spaces $V$ that are $M$-ideals in $\tilde{V}.$ Here $\tilde{V}$ is the order unit space obtained by adjoining an order unit to $V.$ To prove these results, we develop an order theoretic version of the ``Alfsen-Efffros' cone decomposition theorem'' for order smooth $1$-normed spaces.  (As a quick application of this result, we sharpen a result on the extension of bounded positive linear functionals on subspaces of order smooth $\infty$-normed spaces.)
\end{abstract}
\maketitle

\section{Introduction} 

A closed subspace $W$ of a real Banach space $V$ is said to be an  {\it L-summand} if there exists a closed subspace $W'$ of $V$ such that $V = W \oplus_1 W'.$ A closed 
subspace $W$ of a real Banach space $V$ is said to be an $M$-ideal if $W^\perp$ (the annihilator of $W$) is an {\it L-summand} in $V^*.$ The notion of \emph{M}-ideals 
in a Banach space was introduced by E. M. Alfsen and E. G. Effros in \cite{AE} in 1973.  Over the years the \emph{M}-ideals have been studied extensively, resulting 
in a vast theory. This is  an important tool in Functional Analysis. For a comprehensive treatment and for references to the extensive literature on the subject one 
may refer to \cite{HWW} and suitable references therein. 

Though the notion of an \emph{M}-ideal is defined solely in terms of the norm of the Banach space, it has order theoretic germs in its origin. In fact, an \emph{M}-ideal 
in a C$^{\ast}$-algebra is precisely a two sided closed ideal in it. Similarly, an \emph{M}-ideal in an order unit space may be chacterized in terms of the split faces of the 
state space of the order unit space. Both the results are order theoretic in nature. In this paper, we propose to study \emph{M}-ideals in the (non-unital) order smooth 
$\infty$-normed spaces to underline the importance of order unit in the context of \emph{M}-ideals.

In 1951,  R. V. Kadison proved that the self adjoint part of a unital $C^*$-algebra can be represented as the space of continuous affine functions on its state space and thus forms an order unit space \cite{RVK1951}. This observation attracted several mathematicians' interest in the study of convexity theory and ordered normed spaces. In particular, in 1964, D. A. Edwards introduced the notion of base normed spaces \cite{EDW} and A. J. Ellis studied duality between order unit spaces and base normed spaces \cite{ELLIS}. For details one may refer to \cite{Alf} and \cite{GJEM} and references therein.

In 2010, the second author introduced the notion of order smooth $p$-normed spaces ($1\leq p\leq \infty$) so that every order unit space is an order smooth 
$\infty$-normed space and that every base normed space is an order smooth $1$-normed space. He studied the duality between order smooth $p$-normed 
space and order smooth $p'$-normed space, where $\frac{1}{p}+\frac{1}{p'}=1$  \cite{Karn2}. This generalizes the duality between order unit 
space and base normed space established by A. J. Ellis and thus offers a duality theory for non-unital ordered normed spaces. 

In this paper, we obtain an order theoretic version of the `Alfsen-Efffros' cone 
decomposition theorem \cite[Theorem 2.9]{AE} for order smooth $1$-normed spaces satisfying condition $(OS.1.2).$ As an immediate application of this result, we sharpen a 
result on the extension of bounded positive linear functionals on subspaces of order smooth $\infty$-normed spaces \cite[Theorem 4.3]{Karn2}. We give a characterization 
of $M$-ideals in order smooth $\infty$-normed spaces by extending the notion of split faces of the state space to those of the quasi-state space. This result is a generalization of its counterpart for order unit spaces studied by Alfsen and Effros \cite{AE}. At the end of the paper, we prove that an order smooth $\infty$-normed spaces $V$ is an approximate order unit space if and only if $V$ is an $M$-ideal in $\tilde{V}.$ Here $\tilde{V}$ is the order unit space obtained by adjoining an order unit to $V.$ 

The plan of the paper is as follows. In Section 2, we obtain an order theoretic version of the `Alfsen-Efffros' cone decomposition theorem \cite[Theorem 2.9]{AE} for order 
smooth $1$-normed spaces satisfying condition $(OS.1.2).$ We sharpen a result on the extension of bounded positive linear functionals on subspaces of order smooth 
$\infty$-normed spaces \cite[Theorem 4.3]{Karn2}. In Section 3, we discuss a characterization for $M$-ideals in order smooth $\infty$-normed spaces. In Section 4, we characterize approximate order unit spaces as those order smooth $\infty$-normed spaces $V$ that are $M$-ideals in $\tilde{V}.$ 

\subsection{Basic notations and facts}\hfill

Now, we recall some basic notions and facts which will be referred in the paper.
\subsubsection{Faces and cone}\hfill

First, we recall few facts from \cite[Part I]{AE}.   
Let $V$ be a real vector space. A non empty subset $C\subseteq V$ is called a {\it cone} if $\lambda C\subseteq C$ for all $\lambda \geq 0$ and $C+C\subseteq C.$ A cone $C$ is 
called {\it proper} if $C\cap-C=\{0\}.$ The cone $C$ is said to generate $V$, if $V = C - C$. 

A non-empty convex subset $F$ of a convex set $K$ in a real vector space $V$ is called a {\it face} if for  any $x,y\in K,$ we have $x,y \in F$ whenever  $\lambda x+(1-\lambda)y\in F$  
for some $ 0 < \lambda < 1.$ If $K$ is a convex subset of $V,$ then $\cone (K) :=\cup_{\lambda\geq 0}\lambda K$ is the smallest cone containing $S.$ 

Let $S$ be a non-empty subset of a convex set $K$. Then $\face_{K} (S)$ is the smallest 
face of $K$ containing $S$. Thus 
$$\face_K (S) = \cap \{ F: F ~\textrm{is a face of}~ K ~\textrm{containing}~ S \}.$$ 
For $S = \{ x \}$, we write, $\face_K(x)$ for $\face_K(S)$.

Now, let $V$ be a (real) normed linear space and let $V_1$ denote the closed unit  ball of 
$V.$ We say that a cone $C$ in $V$ is {\it facial} if $C = \{0\}$ or $C = \cone (F)$ for some 
proper face $F$ of $V_1.$ Any facial cone is a proper. If $v\in V$ and $v\neq 0,$ then  
$w \in  \face_{V_1} (\frac{v}{\| v \|})$ if and only if $\frac{v}{\| v \|}= \lambda w +(1-\lambda)u$ for some $\lambda \in (0,1)$ and some $u\in V_1.$ We write $C(v) := 
\cone (\face_{V_1}(\frac{v}{\| v \|}))$ for the smallest facial cone containing $v$ if $v \not= 0$. We define $C(0) = \{ 0 \}$.  

For a cone $C$ in $V$, we write 
$$C'=\{v\in V: C\cap C(v)=\{0\}\}.$$ 
It may be noted that $C'$ may not be convex, in general. 

These notions and facts can be found with details in \cite[Part I, Section 2]{AE}. 

The following two results will be used frequently in this paper. 

\begin{lemma}\label{pro-of-facl-cne}\cite[Lemma 2.3]{AE}
	Let $V$ be a normed linear space and let $u_1,\cdots, u_n\in V.$ Then the following facts are equivalent:
	\begin{enumerate}
		\item $u_1,\cdots, u_n\in C(u_1+\cdots+u_n).$
		\item $\|\Sigma_{1=1}^{n}u_i \|=\Sigma_{i=1}^{n} \|u_i \|.$
	\end{enumerate}
\end{lemma}

\begin{theorem}\cite[Part I, Theorem 2.9]{AE}\label{cne-dec}
	Suppose that $C$ is a norm-closed convex cone in a Banach space $V.$ Then every $u\in V$ admits a decomposition $u=v+w, ~\| u \|= \|v \|+ \| v \|,$ where $v\in C$ and $w\in C'.$  
\end{theorem}

\subsubsection{Ordered normed linear spaces} \hfill 

Now, we recall some definitions and facts about (non-unital) ordered normed linear spaces 
studied in \cite{Karn2}.

\begin{definition}\cite{Karn2}
	Let $(V,V^+)$ be a real ordered vector space such that the cone $V^+$ is proper and generating and let $||.||$ be a norm on $V$ such that $V^+$ is closed. For fixed real 
	number $p,1\leq p< \infty,$ consider the following conditions on $V$:
	
	\begin{enumerate}
		\item $(O.p.1)$ For $u,v,w$ with $u\leq v\leq w,$ we have $ \|v\|\leq (||u||^p+||w||^p)^{\frac{1}{p}}.$
		\item $(O.p.2)$ For $v\in V$ and $\epsilon> 0,$ there are $v_1,v_2\in V^+$ such that $v=v_1-v_2 $ and $(||v_1||^p+||v_2||^p)^\frac{1}{p}\leq ||v||+ \epsilon.$
		\item $(OS.p.2)$ For $v\in V,$ there are $v_1,v_2\in V^+$ such that $v=v_1-v_2$ and $||v|| \ge (||v_1||^p+||v_2||^p)^\frac{1}{p}.$
	\end{enumerate}
	For $p=\infty,$ consider the similar conditions on $V$:
	
	\begin{enumerate}
		\item $(O.\infty.1)$ For $u,v,w$ with $u\leq v\leq w,$ we have  $ ||v||\leq \max (||u||,||w||),$
		\item $(O.\infty.2)$ For $v\in V$ and $\epsilon> 0,$ there exist $v_1, v_2\in V^+$ such that $v=v_1-v_2 $ and $\max(||v_1||,||v_2||)\leq ||v||+\epsilon.$
		\item $(OS.\infty.2)$ For $v\in V,$ there are $v_1, v_2\in V^+$ such that $v=v_1-v_2$ and $||v|| \ge \max(||v_1||, ||v_2||).$
	\end{enumerate}
\end{definition} 

\begin{theorem}\cite{Karn2}\label{a1}
	Let $(V,V^+)$ be a real ordered vector space such that the cone $V^{+}$ is proper and generating. Let $||.||$ be a norm on $V$ such that $V^+$ is closed. For each 
	$p,1\leq p\leq \infty,$ we have
	\begin{enumerate}
		\item $||.||$ satisfies $(O.p.1)$ condition on $V$ if and only if $||.||^*$ satisfies the condition $(OS.p'.2)$ on the Banach dual $V^\ast.$
		\item $||.||$ satisfies the condition $(O.p.2)$ on $V$ if and only if $||.||^*$ satisfies the condition $(O.p'.1)$ on $V^\ast.$
	\end{enumerate}
\end{theorem} 

\begin{definition}\cite{Karn2}
	Let $(V,V^+)$ be a real ordered vector space such that the cone $V^+$ is proper and generating and let $||.||$ be a norm on $V$ such that $V^+$ is closed. For a fixed $p,$ 
	$1\leq p\leq  \infty,$ we say that $V$ is  an order smooth $p$-normed space, if $||.||$ satisfies the conditions $(O.p.1)$ and $(O.p.2)$ on $V.$
\end{definition}
The next follows from Theorem \ref{a1} and \cite[Remark 2.3(d)]{Karn2}. 
\begin{theorem}\label{duality-of-p-thry}\cite{Karn2} 
	Let $(V,V^+)$ be a real ordered vector space such that the cone $V^+$ is proper and generating and let $||.||$ be a norm on $V$ such that $V^+$ is closed. For a fixed 
	$p, 1\leq p\leq\infty,$ $V$ is an order smooth $p$-normed space if and only if its Banach dual $V^*$ is an order smooth $p'$-normed space satisfying the condition 
	$(OS.p'.2).$ 
\end{theorem}
\subsubsection{Unital ordered spaces} \hfill 

Let $V$ be a (real) ordered space. An increasing net $\{e_\lambda\}\in V^+$ is called an approximate order unit for $V$ if for each $v \in V$ there is $r > 0$ such
that $r e_\lambda \pm v \in V^+$ for some $\lambda.$ In this case $\{e_\lambda \}$ determines a semi norm $||.||_a$ on $V$ that satisfies $(O.\infty.1)$ and
$(O.\infty. 2).$ More precisely, for $v \in V$, we define  
$$\Vert v \Vert_a = \inf \{ r > 0: r e_{\lambda} \pm v \in V^+~ \textrm{for some}~ \lambda 
\}.$$ 
We call $(V, \{e_\lambda\})$ an \emph{approximate order unit space}  if  $||.||_a$ is a norm 
on $V$ in which $V^+$ is closed. When $e_{\lambda} = e$ for each $\lambda$, $e$ is called 
an \emph{order unit}. Let $(V,V^+)$ be a real ordered vector space, with an order unit $e.$ 
We say that $V^+$ is  \emph{Archimedean}, whenever $v\in V$ with $r e + v\geq 0$ for all $r>0$ implies $v \in V^+.$ In this case, the semi norm $\Vert\cdot\Vert_e$, determined by $e$ as above, is a norm for which $V^+$ is closed. This space is called an \emph{order 
unit space} and is denoted by $(V, e)$.

Next, we describe the dual notion. Let $(V,V^+)$ be a real ordered 
vector space such that $V^+$ is proper and generating. A nonempty convex subset $B$ of 
$V^+$ is called a \emph{base} for $V^+$ if every $u \in V^+,$ $u \neq 0$ has a unique 
representation $u = \lambda u_0,$ where $u_0 \in B$ and $\lambda$ is a positive real 
number. For $v \in V$, we define 
$$\Vert v \Vert_B = \inf \{ r > 0: v \in r \co (B \cup -B) \}.$$
Then $\Vert\cdot\Vert_B$ is a semi norm on $V$. In case, when $\Vert\cdot\Vert_B$ is a 
norm on $V$ with $V^+$ is norm closed, we call $V$ a \emph{base normed space} and 
denote it by $(V, B)$. As $B$ is convex, we note that $\Vert\cdot\Vert_B$ is additive on 
$V^+$. 

We characterize an approximate order unit space among order smooth $\infty$-normed 
spaces.
\begin{proposition}\label{osi-aou}
	Let $V$ be an order smooth $\infty$-normed space. Then $V$ is an approximate order 
	unit space if and only if $S(V)$ is convex.
\end{proposition}
\begin{proof}
	If $(V, \{ e_{\lambda} \})$ is an approximate order unit space, then $S(V)$ is convex. In 
	fact, for any $f \in V^{\ast +}$ we have $\Vert f \Vert = \sup_{\lambda} f(e_{\lambda}).$ 
	Conversely, let $V$ be an order smooth $\infty$-normed space for which $S(V)$ is 
	convex. Then the norm is additive on $V^{\ast +}$. In fact, if $f, g \in V^{\ast +} \setminus 
	\{ 0 \}$, then $f_0 = \Vert f \Vert^{-1} f, g_0 = \Vert g \Vert^{-1} g \in S(V)$. Now, by the 
	convexity, $(\Vert f \Vert + \Vert g \Vert)^{-1} (\Vert f \Vert f_0 + \Vert g \Vert g_0) \in 
	S(V)$. Thus 
	$$\Vert f + g \Vert = \Vert (\Vert f \Vert f_0 + \Vert g \Vert g_0) \Vert = \Vert f \Vert + \Vert g \Vert.$$
	Next, as $V$ is an order smooth $\infty$-normed space, by Theorem 
	\ref{duality-of-p-thry}, $V^{\ast}$ satisfies $(OS.1.2)$. As the condition $(OS.1.2)$ is 
	identical with the condition $1$-generating as defined in \cite{WK}, it follows from 
	\cite[Proposition 9.5]{WK} that $V^{\ast}$ is a base normed space. Now, by 
	\cite[Theorem 9.9]{WK}, $V$ is an approximate order unit space. 
\end{proof}
\section{Cone-decomposition property}

In this section, we shall prove an order theoretic version of the `Alfsen-Effros' cone decomposition Theorem \ref{cne-dec} for order smooth $1$-normed spaces which satisfies $(OS.1.2).$ 
For a subset $W$ of an ordered vector space $(V, V^+)$, we write $W^{+}=W\cap {V^+}.$
\begin{theorem}\label{odr-cone-dec}
	Let $(V,V^+)$ be a complete order smooth $1$-normed space satisfying $(OS.1.2)$ and let $W$ be a closed cone in $V.$ Then for any $v\in V^+,$ there are $w\in W^+$ and $w'\in W'^{+}$ 
	such that $v=w+w'$ and $||v||=||w||+||w'||.$
\end{theorem}

We shall use the following fact to prove Theorem \ref{odr-cone-dec}. 

\begin{lemma}\label{lma-fr-order-cone-dec}
	Let $(V,V^+)$ be an order smooth $1$-normed space satisfying $(OS.1.2).$ If $u\geq 0,$ then $\face_{V_1}(\frac{u}{||u||}) \subseteq V^{+}.$ 
\end{lemma}

\begin{proof}
	Let $u \in V^+.$ Without any loss of generality, we may assume that $\|u\|=1.$ Let $v\in \face_{V_1}(u).$ Then by the definition of $\face_{V_1}(u),$ there exists 
	$w \in V_1$ such that 
	$$u=\lambda v+(1-\lambda)w\mbox{ for some }\lambda\in (0,1).$$ 
	Since $\|u\|=1,$ we have $\|v\|=1=\|w\|.$ Also, as $V$ satisfies $(OS.1.2),$ there exist $v_1,v_2, w_1, w_2 \in V^+$ such that 
	$v =v_1-v_2$ and $w =w_1-w_2$  with  $\|v\| =\|v_1\|+\|v_2\|$ and $\|w\| =\|w_1\|+\|w_2\|.$ 
	Thus $u = u_1-u_2$ where $u_i = \lambda v_i+(1-\lambda)w_i$ for $i = 1, 2.$   
	Since $0\leq u\leq u_1$ and since $V$ is an order smooth $1$-normed space, we have 
	\begin{eqnarray*}
		1 &=& \|u\|\leq \|u_1\|\\
		&\leq &\|\lambda v_1+(1-\lambda)w_1\|\\
		&\leq &\lambda ||v_1||+(1-\lambda)||w_1||\\
		&\leq &\lambda(||v_1||+||v_2||)+(1-\lambda)(||w_1||+||w_2||)\\
		&\leq &\lambda||v||+(1-\lambda)||w|| = 1. \cr
	\end{eqnarray*} 
	Thus $v_2=0=w_2$ so that $v, w \in V^+.$ 
\end{proof}

\begin{proof}[\it Proof of Theorem \ref{odr-cone-dec}]
	Let $W$ be a closed cone of $V$ and $ u\in V^+.$ Then by Theorem \ref{cne-dec}, we have $u=v+w$ with $||u||=||v||+||w||$ for some $v \in W$ and 
	$w\in W'.$ Now, by Lemmas \ref{pro-of-facl-cne} and \ref{lma-fr-order-cone-dec}, we can conclude that $v\mbox{ and }w \in V^+.$ 
\end{proof}
A quick consequence of Lemma \ref{lma-fr-order-cone-dec} is the following:

\begin{corollary}\label{pos-face} 
	Let $(V,V^+)$ be a complete order smooth $1$-normed space satisfying $(OS.1.2).$ Then $face_{V_1}(\frac{u}{||u||}) = face_{V_1^+}(\frac{u}{||u||})$ 
	and $C(u) \subseteq V^+$ whenever $u \in V^+.$ Here $V_1^+ = V_1 \cap V^+.$
\end{corollary}
We also have the following:
\begin{corollary} 
	Let $(V,V^+)$ be a complete order smooth $1$-normed space satisfying $(OS.1.2).$ Then we have $(-V^+)'=V^+, (V^+)'=-V^+.$
\end{corollary}

\begin{proof}
	We shall only prove $(-V^+)' = V^+,$ as similar arguments will work for the other case. Put $W = -V^+$ and let $u \in V^+.$ By Theorem \ref{odr-cone-dec}, we have, 
	$u = v + w$ with $||u|| = ||v|| + ||w||$ for some $v \in W^+, w \in W'^{+}.$ But $W^+ = W \cap V^+ = -V^+ \cap V^+ = \{0\}$ so that $u = w \in W^{'+}.$ Conversely, let $v \in W' 
	:= (-V^+)^{'}.$ Then by the definition, $C(v)\cap (-V^+) = \{0\}.$ Since $V$ satisfies $(OS.1.2),$ there are $v_1, v_2 \in V^+$ such that $v = v_1 - v_2$ and
	$||v|| = ||v_1||+|||v_2||.$ By Lemma \ref{pro-of-facl-cne}, $-v_2\in C(v).$ But $C(v)\cap (-V^+) = \{0\}$ so that $v = v_1 \in V^+.$ 
\end{proof}

We apply Theorem \ref{odr-cone-dec} to sharpen \cite[Theorem 4.3]{Karn2}. Actually, we prove positive and norm preserving extensions of positive bounded linear functionals 
without the assumption that the order smooth subspace be `strong' (\cite[Definition 3.4]{Karn2}). 

\begin{theorem}\label{order-norm-preserving-extension} 
	Let $W$ be an order smooth subspace of an order smooth $\infty$-normed space $(V,V^+).$ Then every positive bounded linear functional on $W$ has a positive norm preserving 
	extension on $V.$   
\end{theorem}
Here by an {\it order smooth subspace} $Y$ of an order smooth $p$-normed space $X$, we mean that $Y$ is also an order smooth $p$-normed space when the order and  the norm of $X$ 
is restricted to $Y$.
\begin{proof}
	Let $f$ be a positive bounded linear functional on $W.$ By the Hahn Banach theorem there exists $F\in V^{*}$ such that $||F|| = ||f||.$ We prove 
	that $F$ is positive. Since $V^{*}$ satisfy $(OS.1.2),$ 
	by Theorem \ref{duality-of-p-thry}, there are $F_1, F_2 \in V^{*+}$ such that 
	$F = F_1 - F_2$ with $||F|| = ||F_1|| + ||F_2||.$ Since $F_1, F_2 \in V ^{*+}$ and $V^*$ is complete, by 
	Theorem \ref{odr-cone-dec}, there are $F_{11}, F_{21}\in W^{\perp +}$ and $F_{12}, F_{22} \in W^{\perp '+}$ such that 
	$F_1 = F_{11} + F_{12}$ with $||F_1|| = ||F_{11}|| + ||F_{12}||$ and $F_2 = F_{21} + F_{22} $ with $||F_2|| = ||F_{21}|| + ||F_{22}||.$  
	Now $F = F_{11} - F_{21} + F_{12} - F_{22},$  where $F_{11}, F_{21} \in W^{\perp +}$ and $F_{12}, F_{22} \in W^{\perp'+}$ such that
	$||F|| = ||F_{11}|| + ||F_{21}|| + ||F_{12}|| + ||F_{22}||.$ If $f_{ij} =\left.F_{ij}\right|_W$ for all $i, j \in \{1, 2\}$ . Then $f_{11} = f_{21} = 0,$ 
	so that $f = f_{12} - f_{22}.$ Further, as $f$ is positive, we have $0\leq f\leq f_{12}.$ Thus by $(O.1.1)$ on $V^*,$ we get $\|f\| \le \|f_{12}\|.$ Therefore,
	\begin{eqnarray*}
		||f||& \leq & ||f_{12}||\\
		& \leq & ||F_{11}|| + ||F_{21}|| + ||F_{12}|| + ||F_{22}||\\
		& = & ||F|| = ||f||
	\end{eqnarray*}
	and consequently, $F_{11} = F_{21} = F_{22} = 0.$ Hence $F = F_{12} \in V^{*+}.$ 
\end{proof}

\section{$M$-ideals in order smooth $\infty$-normed spaces} 

We begin with a  characterization of $M$-ideals in a complete approximate order unit space 
due to Alfsen and Effros \cite{AE}. First, we recall the following notion.
\begin{definition}
	Let $V$ be a normed linear space and let $K$ be a non-empty, closed, convex set in $V.$
	A proper face $F$ of $K$ is said to be a split face of $K$ if $F_K^{C}$ is a proper face of $K$ such that $K = F \oplus_c F_K^{C}.$ Here 
	$$F_K^{C} = \cup \{ \face_K(v) : v \in K ~\textrm{and} ~ \face_K(v) \cap F = \emptyset \}$$ 
	and by $K = F \oplus_c F_K^{C},$ we mean that for each $v \in K$ there exist unique $u\in F, w \in F_K^{C}$ and $\lambda \in [0, 1]$ such that 
	$v = \lambda u + (1 - \lambda ) w.$
\end{definition}
\begin{theorem}\cite[Corollary 5.9, Part II]{AE}\label{M-ideal and splits face in A(K)}
	Let $W$ be a closed subspace of a complete approximate order unit space 
	$(V, \{ e_{\lambda} \}).$ Then $W$ is an $M$-ideal in $V$ if and only if $W^{\perp}\cap 
	S(V)$ is a closed split face of $S(V),$ where $S(V) = \{ f \in V^{\ast +}: \Vert f \Vert = 1 \}$ is the state space of $V.$
\end{theorem}
In this section, we shall prove an analogue of this result for complete order smooth 
$\infty$-normed spaces. We have just noted that in general, in an order smooth 
$\infty$-normed space $V$, $S(V)$ may not be convex. To overcome this situation, we 
present an alternative form of Theorem \ref{M-ideal and splits face in A(K)}. 
For breivity, we shall adopt the following convention: Let $V$ be 
an order smooth $1$-normed space and let $C$ and $D$ be subsets of $V^{+}.$ We shall 
write $V^{+}=C\oplus_{1} D,$ if for $v \in C$ and $w \in D$, we have 
$\Vert v + w \Vert = \Vert v \Vert + \Vert w \Vert$ and if every element $u$ of $V^+$ can be 
written uniquely as $u = v + w$ with $v \in C$ and $w \in D.$ 
\begin{proposition}\label{alt M-ideal aou}
	Let $V$ be an approximate order unit space and  $W$ be a closed subspace of $V$. 
	Then $W^{\perp}\cap S(V)$ is a split face of $S(V)$ if and only if the following conditions 
	hold:
	\begin{enumerate}
		\item $W^{\perp'+}$ is convex;
		\item $V^{*+}= W^{\perp+}\oplus_{1}W^{\perp'+}$.
	\end{enumerate} 
\end{proposition}
\begin{proof}
	Let us observe that 
	$$(\#) \qquad (W^{\perp +} \cap S(V))_{S(V)}^C = W^{\perp' +} \cap S(V).$$ 
	To see  this, we let $f \in W^{\perp'+} \cap S(V)$. Then $C(f) \cap W^{\perp} 
	= \{ 0 \}$ with $\Vert f \Vert = 1$. Then by the definition of $C(f)$ and following 
	Corollary  \ref{pos-face}, we may deduce that $\face_{S(V)} (f) \cap W^{\perp +} = 
	\emptyset$. Thus $f \in (W^{\perp +} \cap S(V))_{S(V)}^C$. Now tracing back the proof, we may conclude that (\#) holds. 
	
	Now first, we assume that $W^{\perp}\cap S(V)$ is a split face of $S(V)$. We show that 
	conditions (1) and (2) hold. Now, we prove (1). For this let 
	$f, g\in W^{\perp'+}, \alpha\in (0,1).$ Then $\Vert f \Vert^{-1} f, \Vert g \Vert^{-1} g \in 
	W^{\perp' +} \cap S(V)$. Thus by the convexity of  $W^{\perp' +} \cap S(V) = (W^{\perp +} \cap S(V))_{S(V)}^C$, we get  
	$$(\alpha \Vert f\Vert + (1-\alpha)\Vert g\Vert)^{-1} \{ \alpha \Vert f\Vert (\Vert f\Vert^{-1}f)+ (1-\alpha) \Vert g \Vert (\Vert g\Vert^{-1}g) \} \in W^{\perp' +} \cap S(V).$$
	Therefore, $\alpha \Vert f\Vert (\Vert f\Vert^{-1}f)+ (1-\alpha) \Vert g \Vert (\Vert g\Vert^{-1}g) \in W^{\perp' +}$ so that (1) holds. To prove (2), let $f \in V^{\ast +} \setminus \{ 0 \}$. Then $\Vert f \Vert^{-1} f \in S(V)$. Since  $W^{\perp +} \cap S(V)$ is 
	a split face of $S(V)$, we have $S(V) = W^{\perp +} \cap S(V) \oplus_c (W^{\perp +} \cap 
	S(V))_{S(V)}^C = W^{\perp +} \cap S(V) \oplus_c (W^{\perp' +} \cap S(V)$. Thus there exist 
	unique $g_0 \in W^{\perp +} \cap S(V)$ and $h_0 \in W^{\perp' +} \cap S(V)$ and 
	$\lambda \in [0, 1]$ such that $\Vert f \Vert^{-1} f = \lambda g_0 + (1 - \lambda) h_0$. 
	Then $f = g + h$ where $g = \lambda \Vert f \Vert g_0 \in W^{\perp +}$ and 
	$h = \lambda \Vert f \Vert h_0 \in W^{\perp' +}$.
	
	Next, assume that conditions (1) and (2) hold. We show that $F = W^{\perp +} \cap 
	S(V)$ is a split face of $S(V)$. Put $G=W^{\perp'+}\cap S(V).$ Since $W^{\perp+}, W^{\perp'+}$ are faces of $V^{*+},$ by Proposition \ref{face}, we get that $F$ and $G$ are faces of $S(V).$  Also, as $W^{\perp} \cap W^{\perp'}=\{0\},$ we may conclude that $F\cap G=\emptyset$.  
	We shall prove that $S(V)=F\oplus_{c} G.$ It suffices to show that $S(V)\subseteq F\oplus_{c} G.$ Let $f\in S(V)\setminus F\cup G \subseteq V^{*+}.$ By (2), there exist unique $g_{0}\in W^{\perp+}$ and $h_{0}\in W^{\perp'+}$ such that $f=g_{0}+h_{0}.$
	Thus $g=\frac{g_{0}}{\Vert g_{0}\Vert}\in F, h=\frac{h_{0}}{\Vert h_{0}\Vert}\in G.$ Also $\Vert g_{0}\Vert+ \Vert h_{0}\Vert =1$ so that $\Vert g_{0}\Vert g+ \Vert h_{0}\Vert h \in F\oplus_{c} G.$ Therefore, $S(V)\subseteq F\oplus_{c}G = F \oplus_c F_{S(V)}^C$ by (\#). 
\end{proof}
\begin{remark}\label{characterization of M-ideal in au space}
	Let $V$ be a complete approximate order unit space and $W$ be a closed subspace of 
	$V.$ Then $W$ is an $M$-ideal in $V$ if and only if $W$ satisfies the following 
	conditions: 
	\begin{enumerate} 
		\item $W^{\perp'+}$ is convex.
		\item $V^{\ast +} = W^{\perp+}\oplus_{1} W^{\perp'+}.$
	\end{enumerate}
\end{remark}
Now we prove the main result of this section. 
\begin{theorem}\label{characterisation of M-ideal by positive cone order smooth infinite normed space}
	Let $V$ be a complete order smooth $\infty$-normed space and $W$ be a closed subspace of $V.$ Then $W$ is an $M$-ideal in $V$ if and only if $W$ satisfies the following conditions.
	
	\begin{enumerate} 
		\item $W^{\perp'+}$ is convex.
		\item $V^{\ast +} = W^{\perp+}\oplus_{1} W^{\perp'+}.$
		%\item If $f\in W^{\perp+}$ and $g\in W^{\perp'+},$ then $||f+g||=||f||+||g||.$
		%\item $W^{\perp}$ satisfies $(OS.1.2).$
	\end{enumerate}
\end{theorem}   
We shall use the following results to prove Theorem \ref{characterisation of M-ideal by positive cone order smooth infinite normed space}. We begin with the following observation.
\begin{proposition}\label{face}
	Let $V$ be an order smooth $\infty$-normed space and  $W$ be a closed subspace of $V$ such that following conditions hold:
	\begin{enumerate}
		\item $W^{\perp'+}$ is convex;
		\item $V^{*+}= W^{\perp+}\oplus_{1}W^{\perp'+}$.
	\end{enumerate} 
	Then $W^{\perp +}$ and $W^{\perp' +}$ are faces of $V^{\ast +}$. 
\end{proposition}
\begin{proof}
	Let $f_{1}, f_{2}\in V^{*+}$ with $f=\alpha f_{1}+(1-\alpha)f_{2}\in W^{\perp+}$ for some $\alpha\in (0,1).$ By (2), we have $f_{1}= g_{1}+ h_{1},$
	$f_{2}=g_{2}+ h_{2}$ for some unique $g_{1}, g_{2} \in W^{\perp+}$ and $h_{1},h_{2}\in W^{\perp'+}$. Put $g=\alpha g_{1}+(1-\alpha)g_{2}$ and $h=\alpha h_{1}+(1-\alpha)h_{2}.$ Since $W^{\perp+}$ and $W^{\perp'+}$ are convex,
	we have $g\in W^{\perp+}$ and $h\in W^{\perp'+}.$ Then $f=g+h$ is a decomposition of $f$ in $W^{\perp+}\oplus_{1}W^{\perp'+}$. As $f\in W^{\perp+},$ by the uniqueness of decomposition, we may conclude that $h=0.$ Thus 
	$h_{1}=0=h_{2}$ so that $W^{\perp+}$ is a face of $V^{*+}.$ Now, by symmetry, $W^{\perp'+}$ is also a face of $V^{*+}.$ 
\end{proof} 
For the next result, we use the following notion defined in \cite{AE}. 
\begin{definition}
	Let $V$ be a normed linear space. For $u, v \in V$ we define $u \prec v,$ if $||v|| = ||u|| + ||v-u||$ (or equivalently, $u \in C(v)$). A subspace $W$ of $V$ is said to be \emph{ hereditary}, if $u \prec v$  with $v \in W$ implies $u \in W.$ 
\end{definition}
It follows from \cite{AE} that this relation is transitive.
\begin{proposition}\label{hereditary1}
	\begin{enumerate}
		\item Let $(V, B)$ be a complete base normed space and let $W$ be a closed 
		subspace of $V$. If $W$ is hereditary, then $W \cap B$ is a face of $B$, or 
		equivalently, $W \cap V_1^+$ is a face of $V_1^+$.
		\item Let $V$ be a complete order smooth $1$-normed space and $W$ be a closed subspace of $V$ such that $W\cap V^{+}_{1}$ be a face in $V^{+}_{1}.$ Then $W$ is hereditary.
		\item Let $W$ be a hereditary subspace of a complete order smooth $1$-normed space $V.$ If $V$ satisfies $(OS.1.2),$ then so does $W.$
	\end{enumerate}
	
\end{proposition}

\begin{proof}
	(1): Let $u_{1}, u_{2}\in B$ such that $u=\alpha u_{1}+(1-\alpha){u_{2}}\in W \cap B$ 
	for some $\alpha\in (0,1).$ Then $\alpha u_{1}, (1 - \alpha ) u_2 \leq u.$ Since $V$ is a 
	base normed space, we get $\Vert u \Vert= \Vert \alpha u_1 \Vert + \Vert  (1-\alpha) u_{2} 
	\Vert$. Thus $\alpha u_{1} \prec u, (1-\alpha)u_{2} \prec u.$ Since $W$ is hereditary, we 
	have $\alpha u_{1}, (1-\alpha)u_{2}\in W.$ Thus $u_{1}, u_{2}\in W$ so that $u_{1}, 
	u_{2}\in W \cap B.$ Hence $W \cap B$ is a face of $B.$ 
	
	(2): Let $v\in W$ and $u\in C(v).$ We shall prove that $u\in W.$ By $(OS.1.2)$ property of $V,$ we have $u=u_1-u_2$ with $||u||=||u_1||+||u_2||,$ for some $u_1,u_2\in V^+.$ 
	Thus $u_1\prec u\prec v.$ Since $W^+$ is a cone in $V,$ by \cite[Theorem 2.9]{AE}, we have $v=v_1+v_2,$ with 
	$||v||=||v_1||+||v_2||$ and $u_1\prec v_1$ where $v_1\in W^+$ and $v_2\in W^{+'}.$ 
	Since $u_{1}\prec v_{1},$ we also have $v_{1} - u_{1} \prec v_{1}.$ Then $v_{1}-u_{1}\in 
	C(v_{1})$. By Corollary \ref{pos-face}, we have $C(v_{1})\subseteq V^{+}$ so that 
	$v_{1}-u_{1}\in V^{+}$. Further, as the norm is additive on $V^+,$ the right hand side of the expression 
	$$\frac{v_1}{||v_1||}=\frac{||u_1||}{||v_1||}\left( \frac{u_1}{||u_1||} \right) + \frac{||v_1-u_1||}{||v_1||} \left( \frac{v_1-u_1}{||v_1-u_1||} \right)$$
	is a convex combination of $\frac{u_1}{||u_1||},$ $\frac{v_1-v_1}{||v_1-u_1||}$ in $V_1^+.$  
	Since $W \cap V_1^+$ is a face of $V_1^+$ and since $\frac{u_1}{||u_1||} \in W \cap V_1^+,$ we have 
	$u_1\in W^+.$ By a similar argument, we can show that $u_2\in W^+.$ Hence $u \in W.$ 
	
	(3): Let $w \in W,$ then by $(OS.1.2)$ property of $V,$ there are $u, v \in V^+$ such that $w = u - v$ and $||w||=||u||+||v||.$ Therefore $u, -v \prec w,$ so by definition of hereditary subspaces, $u, -v\in W.$  Thus $u, v \in W^+$ so that $W$ also satisfies $(OS.1.2).$ 
\end{proof} 

\begin{proof}[{\bf Proof of Theorem \ref{characterisation of M-ideal by positive cone order smooth infinite normed space}}] \hfill 
	
	Let $W$ be an $M$-ideal in $V.$ Then $W^{\perp}$ is an $L$-summand of $V^{\ast}$ so that $W^{\perp'}$ is also an $L$-summand of 
	$V^{\ast}$ with $V^{\ast} = W^{\perp} \oplus_1 W^{\perp'}.$ Thus $W^{\perp' +} = W^{\perp'} \cap V^{\ast +}$ is convex. Also, 
	by the order cone decomposition Theorem \ref{odr-cone-dec}, condition (2) holds. 
	
	Conversely, assume that conditions (1) and (2) hold. Let $f \in V^{*+}.$ Then by condition $(2),$ there exist unique $g\in W^{\perp +},$ 
	and $h\in W^{\perp'+}$ such that $f=g+h$ with $||f||=||g||+||h||.$ Let us write $L_0(f)=g.$ Then by the uniqueness of decomposition 
	$L_{0}:V^+\mapsto V^+$ is well defined and $L_{0}(\alpha f)=\alpha L_0(f)$ for all $\alpha \geq 0.$ Now, let $f_1, f_2\in V^+.$ Again applying  
	condition (2), we can find unique $g_1, g_2 \in W^{\perp +}$ and $h_1, h_2 \in W^{\perp ' +}$ such that $f_{i}=g_{i}+h_{i}$ and 
	$||f_i||=||g_{i}||+||h_{i}||$ for $i = 1, 2.$ Then $f_1+f_2=(g_1+g_2)+(h_1+h_2),$ where $g_1+g_2\in W^{\perp+}$ and 
	$h_1+h_2\in W^{\perp'+}$ by the condition (1). Thus by the condition (2), we have $||f_1+f_2||=||(g_1+g_2)||+||(h_1+h_2)||$ so that 
	$L_{0}(f_1+f_2)=L_{0}(f_1)+L_{0}(f_2).$ Now, let $f \in V^{\ast}.$ By the condition $(OS.1.2)$ in $V^{\ast},$ there are $f_1, f_2 \in V^{\ast +}$ such that 
	$f = f_1 - f_2$ with $||f|| = ||f_1|| + ||f_2||.$ Let us write $L(f)=L_{0}(f_1)-L_{0}(f_2).$ As $L_0$ is additive on $V^{\ast +},$ it is routine to check that 
	$L: V^{\ast} \to V^{\ast}$ is  a well defined, positive linear mapping with $L(V^{\ast}) \subset W^{\perp}.$ We prove that $L$ is an \emph{L-projection} onto $W^{\perp}.$ 
	Let $f\in V^{*},$ then by $(OS.1.2)$ in $V^{\ast},$ there are $g,h\in V^{*+}$ such that $f=g-h$ with $||f||=||g||+|||h||.$ Now  
	\begin{eqnarray*} 
		||f|| &\le& ||L(f)|| + ||f-L(f)|| \\
		&=& ||L(g)-L(h)|| + ||g-h-L(g)+L(h)||\\
		&\leq& ||L(g)||+||L(h)||+||g - L(g)||+||h - L(h)||\\
		&=& (||L_{0}(g)||+||g-L_{0}(g)||)+(||L_{0}(h)||+||h -L_{0}(h)||)\\
		&=& ||g||+||h||=||f||
	\end{eqnarray*}
	so that $||L(f)||+||f-L(f)||=||f||$ for all $f\in V^{*}.$ Next, we show that $L(f)=f$ for all 
	$f\in W^{\perp}.$ To see this, let $f\in W^{\perp}.$ Since by Proposition \ref{hereditary1}, 
	$W^{\perp}$ satisfies $(OS.1.2)$, there are $f_1, f_2\in W^{\perp+}$ such that 
	$f=f_1-f_2$ with $||f||=||f_1||+||f_2||.$ Now, by Theorem \ref{odr-cone-dec}, 
	$f_{i}=g_{i}+h_{i},$ where $g_{i}\in W^{\perp+}$ and $h_{i}\in W^{\perp'+}$ for $i= 1,2.$ 
	As $f_{i},g_{i}\in W^{\perp +},$ we have $h_i \in W^{\perp} \cap W^{\perp'+} = \{0\}.$ Thus 
	by the construction, $g_1=L_{0}(f_{1})$ and $g_{2}= L_{0}(f_{2})$ so that 
	$L(f)=L_{0}(f_{1})-L_{0}(f_2)=g_1-g_2=f$ if $f \in W^{\perp}.$ Thus 
	for any $f\in V^{*},$ we have $L^{2}(f)=L(L(f))=L(f).$ Hence $L$ is an $L$-projection of 
	$V^{*}$ onto $W^{\perp}$ and therefore $W$ is an $M$-ideal in $V.$ 
\end{proof} 

If we attempt to get a prototype of Theorem \ref{M-ideal and splits face in A(K)} for order 
smooth $\infty$-normed spaces, we need to replace the ``state spaces'' by the 
``quasi-state spaces'' as the quasi-state space of an order smooth $\infty$-normed space 
is always convex.  Accordingly, we need to `adjust' the definition of split faces as well. 

First, let us note that  if $V$ is an approximate order unit space, then there is a bijective 
correspondence between the class of faces of $S(V)$ and the class of non zero faces of
$Q(V)$ containing zero. And these correspondences is given by 
$$F \subseteq S(V) \mapsto \co(F\cup \{ 0 \})\subseteq Q(V)$$ 
and 
$$G \subseteq Q(V) \mapsto G \cap S(V) \subseteq S(V).$$
\begin{lemma} 
	Let $V$ be an approximate order unit space and let  $F$ be a face of $S(V).$ 
	Then $S(V)\cap (\cone(\co(F\cup \{0\})))'= F^{C}_{S(V)}$.
\end{lemma}
\begin{proof}
	Let $f\in S(V)$. Then
	\begin{align*}
	(\ddagger) \qquad f\in (\cone(\co(F\cup\{0\})))'&\Leftrightarrow C(f)\cap \co(F\cup\{0\})=\{0\}\\
	&\Leftrightarrow \co(\face_{S(V)}(f)\cup \{0\})\cap \co(F\cup \{0\})=\{0\}.\\
	\end{align*}
	Now as, $\face_{S(V)}(f)\subseteq S(V)$ and $F\subseteq S(V)$, we get that $\face_{S(V)}(f)\cap F=\emptyset$. Thus $f\in F^{C}_{S(V)}$, that is, $S(V)\cap (\cone(\co(F\cup\{0\})))'\subseteq F^{C}_{S(V)}$. 
	
	Conversely, let $f \in F_{S(V)}^C$. Then $f \in S(V)$ and $\face_{S(V)}(f)\cap F=\emptyset$. By ($\ddagger$), it suffices to show that $\co(\face_{S(V)}(f)\cup \{0\})\cap \co(F\cup \{0\})=\{0\}$. Let $g\in \co(\face_{S(V)}(f)\cup \{0\})\cap \co(F\cup \{0\})$. Then there exist $g_1 \in \face_{S(V)}(f)$, $g_2 \in F$ and 
	$\lambda, \mu\in [0,1]$ such that $g=\lambda g_{1}=\mu g_{2}$. As $g_{1},g_{2}\in S(V)$, we get $\lambda=\mu$. Now, if $\lambda\neq 0$, then $g_{1} = g_{2}\in \face_{S(V)}(f) \cap F = \emptyset$
	so that $\lambda=0=\mu$ and consequently, $g=0$. Thus $S(V)\cap (\cone(\co(F\cup \{0\})))'= F^{C}_{S(V)}$.
\end{proof}
\begin{definition}Let $V$ be an order smooth $\infty$-normed space.  Let $G$ and $H$ be any two faces of $Q(V)$ containing zero such that $G\cap H=\{0\}$. We define
	$$G\oplus_{c,1}H=\{\lambda g+(1-\lambda)h: g \in G, h \in H, \Vert g\Vert = \Vert h\Vert, \lambda \in [0,1]\}.$$ 
	For a face $G$ of $Q(V)$ containing zero, we shall say that $G$ is a \emph{split face} of $Q(V)$, if $G_{Q(V)}' := (\cone(G))' \cap Q(V)$ is also a face of $Q(V)$ (containing zero) and if every element in $Q(V)$ has a unique representation in $G\oplus_{c,1} G_{Q(V)}'$. 
\end{definition}
\begin{remark} 
	By the definition of a split face, we get that $\Vert f \Vert \le \Vert g \Vert = \Vert h 
	\Vert$. But we can show that these norms are equal. To see this, let $f\in Q(V) \setminus 
	\{0\},$ Then $f_{1}= \Vert f\Vert^{-1}f\in Q(V)$. Thus there exist unique $g_{1}\in 
	W^{\perp}\cap Q(V), h_{1}\in (W^{\perp}\cap Q(V))'_{Q(V)}$ with $\Vert g_1 \Vert = \Vert 
	h_1 \Vert$ such that $f_1 = \lambda g_1 + (1-\lambda) h_1$ for some $\lambda 
	\in [0,1].$  Now 
	$$1=\Vert f_1 \Vert \leq \lambda \Vert g_{1} \Vert +(1-\lambda) \Vert h_{1} \Vert \leq 1.$$
	Thus $f=\lambda g+ (1-\lambda)h$, where $g= \Vert f\Vert g_{1} \in W^{\perp}\cap Q(V)$ 
	and $h= \Vert f\Vert h_{1} \in (W^{\perp}\cap Q(V))'_{Q(V)}$ with $\Vert g\Vert= \Vert 
	f \Vert= \Vert h\Vert$.
\end{remark}

We show that this notion is a extension of a split face of $S(V)$ as follows:
\begin{theorem} Let $V$ be an approximate order unit space and $F$ be a face of $V.$ Then   
	$F$ is a split face of $S(V)$ if and only if $\co(F\cup \{0\})$ is a split face of $Q(V)$ that is $Q(V)= \co(F\cup\{0\})\oplus_{c,1} \co(F\cup\{0\})_{Q(V)}'$.
\end{theorem}
\begin{proof} Let $F$ be split face of $S(V).$ 
	Since $\co(F\cup \{0\})_{Q(V)}'=\co(F^{C}_{S(V)}\cup \{0\})$ and since $F^{C}_{S(V)}$ is a face of $S(V),$ we conclude that $\co(F\cup\{0\})_{Q(V)}'$ is a face of $Q(V).$ Let $f\in Q(V)\setminus \{0\}.$ Then $\Vert f\Vert^{-1}f\in S(V).$
	Since $F$ is a split face of $S(V).$ There exist an unique element $g_{0}\in F, h_{0}\in F^{C}_{S(V)}$ such that $\Vert f\Vert^{-1}f=\lambda g_{0}+ (1-\lambda)h_{0}$ for some $\lambda\in [0,1].$
	Put $g= \Vert f\Vert g_{0}, h=\Vert f\Vert h_{0}$. Then $f=\lambda g + (1-\lambda) h$ where $\Vert g\Vert= \Vert h\Vert=\Vert f\Vert$ and $g\in \co(F\cup \{0\})$ and $h\in \co(F^{C}_{S(V)}\cup\{0\})=\co(F\cup\{0\})_{Q(V)}'.$ To prove uniqueness,
	let $g_{1}\in \co(F^{C}_{S(V)}\cup \{0\})$  and $h_{1}\in \co(F\cup \{0\})_{Q(V)}'$ such that $\Vert g_{1}\Vert= \Vert h_{1}\Vert= \Vert f \Vert$ and $f=\mu g_{1}+(1-\mu)h_{1}$ for some $\mu \in [0,1].$ Then
	we have $\Vert f\Vert^{-1}f= \mu \Vert f\Vert^{-1}  g_{1}+(1-\mu)\Vert f\Vert^{-1} \Vert h_{1}\Vert.$ Now by uniqueness of split face we have $\Vert f\Vert^{-1} g_{1}= g_{0}=\Vert f\Vert^{-1} g.$ Thus we have $g_{1}=g$ and similarly,
	we have $h_{1}=h$. Hence $\co(F\cup \{0\})$ is a split face of $Q(V).$
	
	Conversely, let $\co(F\cup \{0\})$ is a split face of $Q(V).$ We have to show that $F$ is a split face of $S(V).$ Since $\co(F\cup\{0\})$ is a split face of $Q(V),$ $\co(F\cup \{0\})_{Q(V)}'$ is also a face of $Q(V).$ But $\co(F \cup \{0\})_{Q(V)}' =\co(F_{S(V)}^{C} \cup \{0\}).$ Thus $F^{C}_{S(V)}$ is a face of $S(V)$ as well. Let $f\in S(V)\subseteq Q(V)=\co(F\cup \{0\})\oplus_{c,1} \co(F\cup \{0\})_{Q(V)}'.$ Then there exist unique pair $g\in \co(F\cup \{0\}), h\in \co(F^{C}_{S(V)}\cup \{0\})$ with $\Vert g\Vert=\Vert h\Vert=\Vert f\Vert=1$ (so that $g,h\in S(V)$) such that $f=\lambda g+ (1-\lambda)h$ for some $\lambda \in [0,1].$ Thus $g\in \co(F\cup \{0\})\cap S(V)= F$ and $h\in \co(F^{C}_{S(V)}\cup \{ 0 \})\cap S(V)=F^{C}_{S(V)}.$ Thus $f\in F\oplus_{c}F^{C}_{S(V)}.$
\end{proof}
\begin{lemma}\label{dual face}
	Let $V$ be a complete order smooth $\infty$-normed space and let $W$ be a 
	closed subspace of $V.$ Then $W^{\perp+'}\cap Q(V)= W^{\perp'}\cap Q(V)$.
\end{lemma}

\begin{proof}
	Since $\cone(W^{\perp} \cap Q(V))= W^{\perp+}$, we have 
	$$(W^{\perp}\cap Q(V))'_{Q(V)}= (\cone(W^{\perp}\cap Q(V)))'\cap Q(V)= W^{\perp+'} \cap Q(V).$$
	We show that $W^{\perp+'}\cap Q(V)= W^{\perp'}\cap Q(V)$. If $f\in Q(V)$, then
	$C(f)\subseteq V^{\ast +}$, by Corollary \ref{pos-face}. Thus $C(f)\cap W^{\perp+}=C(f)\cap W^{\perp}$. Since $W^{\perp'} = \{ f \in V^{\ast}: C(f) \cap W = \{ 0 \} \}$, it follows that $W^{\perp+'}\cap Q(V)= W^{\perp'}\cap Q(V)$.
\end{proof} 
\begin{proposition}\label{M-ideal-Q(V)}
	Let $V$ be a complete order smooth $\infty$-normed space and let $W$ be a closed subspace of $V.$ Then $W$ is an $M$-ideal in $V$ if and only if $W^{\perp} \cap Q(V)$ is a split face of $Q(V)$.
\end{proposition}
\begin{proof} 
	First, let us assume that $W$ is an $M$-ideal in $V$. Then by Theorem \ref{characterisation of M-ideal by positive cone order smooth infinite normed space}, $W^{\perp'+}$ is convex and $V^{*+}= W^{\perp+}\oplus_{1} W^{\perp'+}$. Also, by Lemma \ref{dual face}, we have $(W^{\perp}\cap Q(V))_{Q(V)}'=W^{\perp'}\cap Q(V)$.
	We show that $W^{\perp'}\cap Q(V)$ is a  face of $Q(V)$. Let $f_{1}, f_{2}\in Q(V)$ be  such that $f=\alpha f_{1}+(1-\alpha)f_{2}\in W^{\perp'}\cap Q(V)$ for some $\alpha\in (0,1)$. %Then $f\in W^{\perp'+}.$ 
	As $V^{*+}= W^{\perp+}\oplus_{1} W^{\perp'+}$, there are unique $g_{1}, g_{2}
	\in W^{\perp+}, h_{1}, h_{2}\in W^{\perp'+}$ such that $f_{i}= g_{i}+ h_{i},$ with $\Vert f_{i}\Vert= \Vert g_{i}\Vert +\Vert h_{i}\Vert $, for $i=1,2$. Then $f=(\alpha g_{1}+(1-\alpha)g_{2})+(\alpha h_{1}+ (1-\alpha)h_{2})= g+h,$ where
	$g=\alpha g_{1}+(1-\alpha)g_{2}$ and $h=\alpha h_{1}+ (1-\alpha)h_{2}.$ As $W^{\perp+}$ and $W^{\perp'+}$ are convex, we get that $g\in W^{\perp+}, h\in W^{\perp'+}.$ Next, as $W$ is an $M$-ideal of $V,$ $W^{\perp'}$ is an $L$-summand of $V^{\ast}$ so that $W^{\perp'+} - W^{\perp'+} \subseteq W^{\perp'}$. Thus $g = f - h \in W^{\perp'+} - W^{\perp'+} \subseteq W^{\perp'}$ so that $g\in W^{\perp'} \cap W^{\perp} = \{ 0 \}$. Therefore, $g_{1}=g_{2}=0$ and consequently, $f_{1} = h_{1}, f_{2} = h_{2}\in W^{\perp'+} \cap Q(V)$ so that $(W\cap Q(V))_{Q(V)}'$ is a face of $Q(V)$. Similarly, we can prove that $W^{\perp} \cap Q(V)$ is also a face of $Q(V)$. 
	
	Now, we show that $W^{\perp} \cap Q(V)$ is a split face of $Q(V)$. For this, let $f \in Q(V)$. Since $Q(V)\subseteq V^{*+} = W^{\perp+}\oplus_{1} W^{\perp'+}$, there are  unique $g_{0}\in W^{\perp+}, h_{0}\in W^{\perp'+}$ such that $f= g_{0}+ h_{0}$ and  $\Vert f\Vert=\Vert g_{0}\Vert+ \Vert h_{0}\Vert$. Put $g= \Vert f\Vert \Vert g_{0}\Vert^{-1}g_{0}, h= \Vert f\Vert \Vert h_{0}\Vert^{-1}h_{0}$. Then $g\in W^{\perp}\cap Q(V)$ and $h\in W^{\perp'+}\cap Q(V))=(W^{\perp} \cap Q(V))'_{Q(V)}$ and consequently, 
	$$f=(\Vert g_{0}\Vert \Vert f\Vert^{-1})  g+(\Vert h_{0}\Vert \Vert f\Vert^{-1}) h\in W^{\perp}\cap Q(V)\oplus_{c,1} (W^{\perp}\cap Q(V))'_{Q(V)}.$$ 
	Hence $W^{\perp}\cap Q(V)$ is a split face of $Q(V)$.
	
	Conversely, assume that $W^{\perp}\cap Q(V)$ is a split face of $Q(V)$. We show that $W^{\perp'+}$ is convex and  that $V^{*+}= W^{\perp+}\oplus_{1} W^{\perp'+}$. Let $f, g\in W^{\perp'+}$ and $\alpha\in (0,1)$. Put  $h=\alpha f+ (1-\alpha)g$.  If we put  $\lambda = \max \{ \Vert f \Vert , \Vert g \Vert \}$, $h_{0} = \lambda^{-1} h$, $f_{0} = \lambda^{-1} f$ and $g_{0}= \lambda^{-1}g$, then $f_{0}, g_{0}\in W^{\perp'}\cap Q(V)$ with $h_0 = \alpha f_0 + (1-\alpha) g_0$.  Since $W^{\perp}\cap Q(V)$ is a split face of $Q(V)$, $W^{\perp'}\cap Q(V)= (W^{\perp}\cap Q(V))'_{Q(V)}$ is convex.  Thus $h_{0} \in W^{\perp'} \cap Q(V)$ and consequently, $h \in W^{\perp'+}$. Therefore, $W^{\perp'+}$ is convex.
	
	Finally, let $f\in V^{*+}\setminus \{0\}$. Then $f_{1} = \Vert f\Vert^{-1}f \in Q(V) = W^{\perp} \cap Q(V) \oplus_{c,1} W^{\perp'} \cap Q(V)$. Thus there exist unique $g_{1}\in W^{\perp}\cap Q(V)$, $h_{1}\in W^{\perp'}\cap Q(V)$ and $\alpha\in [0,1]$ such that $f_{1}= \alpha g_{1}+(1-\alpha)h_{1}$ with $\Vert g_{1} \Vert= \Vert h_{1}\Vert= \Vert f_{1} \Vert=1$. Then $f= g+h$, where $g = \alpha \Vert f \Vert g_{1} \in W^{\perp+}$ and  $(1-\alpha) \Vert f \Vert h_{1} \in W^{\perp'+}$. Also  
	$$\Vert f\Vert =\Vert g+ h\Vert\leq \Vert g\Vert+\Vert h\Vert \leq \alpha \Vert f \Vert + (1-\alpha) \Vert f\Vert= \Vert f\Vert$$ 
	so that $\Vert f\Vert= \Vert g\Vert + \Vert h\Vert$. This completes the proof.
\end{proof}

\section{\bf $M$-ideals and adjoining of an order unit}

Let $V$ be an order smooth $\infty$-normed space and consider $\tilde{V}=V\oplus \mathbb{R}$. If we define $\tilde{V}^+ = \{ (v,\alpha): l_{V}(v) \leq \alpha \}$ where $l_{V}(v) = \inf \{ \Vert u \Vert: u, u + v \in V^+ \}$, then $(\tilde{V}, \tilde{V}^+)$ becomes a real ordered vector space. In this case, $(0, 1)$ acts as an order unit and $\tilde{V}^+$ is Archimedean so that $(\tilde{V}, (0, 1))$ is an order unit space. Moreover, $v \mapsto (v, 0)$ is an isometric order isomorphic embedding of $V$ in $(\tilde{V}, (0, 1))$. Further, $(\tilde{V}, (0, 1))$ is determined uniquely by $V$ upto a unital order isomorphism in such a way that $V$ has co-dimension $1$ in $(\tilde{V}, (0, 1))$. For a detailed information one can see \cite[Section 4]{Karn2}. In this section, we obtain the conditions under which $V$ is an $M$-ideal in $\tilde{V}.$ The following result (due to Alfsen and Effros) will be used for this purpose. Throughout this section, we shall assume that all order normed 
spaces are (norm) complete.

\begin{theorem}\label{characterization of order unit space} \cite[Theorem 6.10]{AE}
	Let $(V, e)$ be an order unit space and let $W$ be a closed subspace of $V.$ Then following sets of statements are equivalent:
	\begin{enumerate}
		\item $W$ is an $M$-ideal.
		\item  $W$ satisfies each of the following conditions:
		
		\begin{enumerate}
			\item  $W$ is positively generated;
			\item $W$ is an order ideal;
			\item $(V/W, \pi(e))$ is an order unit space; 
			\item Given $v,w\in V^+$ and $\epsilon > 0,$
			$$\pi([0,v])\cap \pi([0,w])\subset \pi([0,v+\epsilon e]\cap [0,w+\epsilon e])$$
			where $[x,y]:=\{z\in V: x\leq z\leq y\}$ for $x\leq y$ in $V.$  
		\end{enumerate}
	\end{enumerate}
	Here $\pi: V \to V/W$ is the canonical quotient mapping.
\end{theorem}
We shall apply this result to characterize approximate order unit spaces as those order smooth $\infty$-normed spaces which are $M$-ideals in order unit spaces obtained by adjoining order units to these spaces. First, we prove the sufficient condition.
\begin{theorem}\label{AOU space V is M-ideal in its adjoining order unit}
	Let $(V,V^+,\{e_\lambda\}_{ \lambda \in D})$ be an approximate order unit space, and let $(\tilde{V},\tilde{V}^+)$ be the order unit space obtained 
	by adjoining an order unit to $V.$ Then $V$ is an $M$-ideal in $\tilde{V}.$
\end{theorem}
We shall prove this result in several steps. 
\begin{proposition}\label{essential proposition1}
	Let $V$ be an order smooth $\infty$-normed space and $\tilde{V}$ be the order unit space obtained by adjoining an order unit to $V.$ Then 
	\begin{enumerate}
		\item  $V$ is positively generated.
		\item $V$ is an order ideal in $\tilde{V}.$
		\item $(\tilde{V}/V, \tilde{\pi}((0,1)))$ is an order unit space.
	\end{enumerate}
	Here $\tilde{\pi}: \tilde{V} \to \tilde{V}/ V$ is the natural quotient mapping.
\end{proposition}

\begin{proof}
	Condition (1) follows from the definition of $V$ and condition (2) follows from the construction of $\tilde{V}$ (see \cite[Theorem 4.1]{Karn2}). To prove (3), first note 
	that the natural quotient map $\tilde{\pi}: \tilde{V} \to \tilde{V}/ V$ is positive and that $\tilde{\pi}(0, 1)$ is an order unit  for $\tilde{V}/ V.$ 
	We show that $(\tilde{V}/V)^+$ is an Archimedean. Let $\tilde{\pi}(u, \alpha)\in \tilde V/V$ such that $\tilde{\pi}(u,\alpha)\leq \frac{1}{n}
	\tilde{\pi}(0,1)$ for all $n\in \mathbb{N}.$ Then $\tilde{\pi}(0,\frac{1} {n}-\alpha)=\tilde{\pi}(-u,\frac{1}{n}-\alpha)\geq 0$ so that  
	$\frac{1}{n}-\alpha\geq 0 $ for all $n\in\mathbb{N}.$ Consequently, $\tilde{\pi}(u,\alpha)=\tilde{\pi} (0,\alpha)\leq 0.$  
\end{proof}

\begin{lemma}\label{lemma for V is a M-ideal}
	Let $V$ be an order smooth $\infty$-normed space and $\tilde{V}$ be the order unit space obtained by an order unit to $V.$ If 
	$\tilde{\pi}: \tilde{V} \to \tilde{V}/ V$ is the natural quotient map, then for all $(u,\lambda)\in \tilde{V}^+,$ we have 
	$$\tilde{\pi}[(0, 0), (u,\lambda)] =\{\tilde{\pi}(0,\mu):  0\leq \mu\leq \lambda\}.$$
\end{lemma}

\begin{proof}
	Let us consider the order interval $[(0, 0), (u,\lambda)],$ where $(u,\lambda)\in \tilde{V}^+.$ Now for any $\mu \in [0, \lambda],$ we have $0\leq (\frac{\mu}{\lambda}u, \mu)\leq (u,\lambda).$ Thus 
	$$\tilde{\pi}(0,\mu) = \tilde{\pi}(\frac{\mu}{\lambda}u,\mu)\in \tilde{\pi}[(0,0),(u,\lambda)] .$$ Conversely, let $(x,\mu)$ in $\tilde V $ such that 
	$0\leq (x,\mu)\leq(u,\lambda)$ in $(\tilde{V},\tilde{V}^+).$ Then we have $0\leq \mu\leq \lambda.$ Now the observation 
	$\tilde{\pi}(x,\mu)=\tilde{\pi}(0,\mu)$ completes the proof. 
\end{proof}
\begin{proof}[\bf Proof of Theorem \ref{AOU space V is M-ideal in its adjoining order unit}]
	\hfill 
	
	Let $(u_i,\gamma_i) \in \tilde{V}^+,$ for $i=1,2$ and let $\epsilon >0$. We shall show that 
	\begin{align*} 
	\tilde{\pi}([(0,0),(u_1,\gamma_1)]) &\cap \tilde{\pi}([(0,0),(u_2,\gamma_2)]) \\ 
	&\subset \tilde{\pi}([(0,0),(u_1,\gamma_1+\epsilon)]\cap [(0,0),(u_2,\gamma_2+\epsilon)]).
	\end{align*}
	Since $(u_i,\gamma_i) \in \tilde{V}^+,$ we have $l_V(u_i)\leq \gamma_i$ for $i=1,2$. Thus as $\epsilon> 0$ and as $\{ e_{\lambda} \}$ is an approximate order unit for $V$, there exist $\lambda$ such that 
	$u_i + (\gamma_i+\epsilon)e_{\lambda_i} \in V^+$ for $i = 1, 2$. Put $\gamma = \min\{\gamma_1,\gamma_2\}.$ Then for $i = 1, 2$ we have 
	$$u_i + \gamma e_{\lambda} + ( \gamma_i - \gamma + \epsilon ) e_{\lambda} = u_i + ( \gamma_i + \epsilon ) e_{\lambda} \in V^+$$
	so that $l_V( u_i  + \gamma e_{\lambda}) \le \gamma_i - \gamma + \epsilon$, or equivalently, $( u_i  + \gamma e_{\lambda}, \gamma_i - \gamma + \epsilon) \in  \tilde{V}^+$. Thus $(- \gamma e_{\lambda}, \gamma ) \le (u_i, \gamma_i + \epsilon )$ for $i = 1, 2$. As $\Vert e_{\lambda} \Vert \le 1$, we have $0 \le (e_{\lambda}, 0) \le (0, 1)$ so that $(- \gamma e_{\lambda}, \gamma ) \in \tilde{V}^+$. Hence $(- \gamma e_{\lambda}, \gamma ) \in [(0,0),(u_1,\gamma_1+\epsilon)]\cap [(0,0),(u_2,\gamma_2+\epsilon)]$. Now the result follows from Proposition \ref{essential proposition1}.
\end{proof}   
\begin{remark}
	Let  $X$ be a Banach space, and let $Y$ be a closed subspace of $X$ such that $Y\neq 0.$ It follows from \cite[Proposition 2.2]{AE} that the $M$-ideals of $Y$ 
	are precisely the $M$-ideals of $X$ contained in $Y.$ Thus each $M$-ideal of an approximate order unit space $(V, \{e_\lambda \})$ is an $M$-ideal of $(\tilde{V}, \tilde{V}^+).$ 
\end{remark} 
Now, we proceed to prove the converse of Theorem \ref{AOU space V is M-ideal in its adjoining order unit}. More precisely, we aim to prove that a complete order smooth 
$\infty$-normed space $V$ is an $M$-ideal in $\tilde{V}$, only if $V$ is an approximate 
order unit space. 

\begin{lemma}\label{concept of adjoining order unit}
	Let $V$ be an order smooth $\infty$-normed space and consider $\tilde{V}$, the order unit space obtained by adjoining an order unit to $V.$  Then $S(\tilde{V})$ is affinely homeomorphic to $Q(V).$  
\end{lemma}
\begin{proof}
	Let $g\in Q(V).$ Define $\tilde{g}(v, \alpha ) = g(v) + \alpha$ for $(v, \alpha ) \in \tilde{V}.$ If $(v,\alpha) \in \tilde{V}^+,$ then for $\epsilon>0,$ there exist 
	$u\in V^+$ such that $u+v\geq 0$ and $||u||<\alpha+\epsilon.$ Thus 
	$$\tilde{g}(v,\alpha) = g(v)+\alpha\geq -g(u)+\alpha\geq-||u||+\alpha>-\epsilon.$$
	Since $\epsilon$ is independent of $g,$ we see that $\tilde{g}$ is positive linear map on $\tilde{V}$ with $\tilde{g}(0,1)=1.$ So $\tilde{g}$ in 
	$S(\tilde{V}).$ Further, if  $h \in S(\tilde{V})$ is any extension of $g,$ then $h(0, 1) = 1 = \tilde{g}(0, 1)$ so that $\tilde{g} = h.$ Thus each $g \in Q(V)$ has a unique 
	extension $\tilde{g} \in S(\tilde{V})$ and consequently, we obtain a well defined and bijective map $\phi:Q(V)\mapsto S(\tilde{V})$ by $\phi(f)=\tilde{f},$ where 
	$\tilde{f}(v,\alpha)=f(v)+\alpha.$ Now, it is routine to check that $\phi$ is affine as well as $w^*$-$w^*$ homeomorphism. 
\end{proof}

\begin{theorem}
	Let $V$ be a complete order smooth $\infty$-normed space. Then $V$ is an $M$-ideal in $\tilde{V}$ if and only if $V$ is an approximate order unit space.
\end{theorem}

\begin{proof} 
	If $V$ is an approximate order unit space, then by Theorem \ref{AOU space V is M-ideal in its adjoining order unit}, $V$ is an $M$-ideal in $\tilde{V}.$ Conversely, assume that 
	$V$ be an order smooth $\infty$-normed space such that $V$ is an $M$-ideal in $\tilde{V}$. Note that 
	$$V^{\perp} = \{ f \in (\tilde{V})^{\ast}: f(v, 0) = 0 ~\mathrm{for all}~ v \in V \} = \mathbb{R} \tilde{0}$$ 
	where $\tilde{0} \in S(\tilde{V})$ with $\tilde{0}(v, \alpha ) = \alpha$ (as described in the proof of Lemma \ref{concept of adjoining order unit}). Thus $V^{\perp} \cap 
	S(\tilde{V}) = \{ \tilde{0} \}$. Since $V$ is an $M$-ideal in $\tilde{V}$, by Theorem \ref{M-ideal and splits face in A(K)}, we may conclude that $\{ \tilde{0} \}$ is a split face of $S(\tilde{V})$.
	
	Next, we show that $\{0\}_{S(\tilde{V})}^{C}=\phi(S(V))$ so that $\phi(S(V))$ is also a (split) face of $S(\tilde{V})$ where $\phi: Q(V) \to S(\tilde{V})$ is the affine homeomorphism described in the proof of Lemma \ref{concept of adjoining order unit}.  In fact 
	\begin{align*}
	\{\tilde{0}\}_{S(\tilde{V})}^C &= \{g\in S(\tilde{V}): \face_{S(\tilde{V})}(g)\cap \{\tilde{0}\} = \emptyset \}\\
	&=\{g\in S(\tilde{V}): \tilde{0}\notin  \face_{S(\tilde{V})}(g)\}\\
	&=\{\phi(f):f\in Q(V), 0\notin \face_{Q(V)}(f) \}\\
	&=\{\phi(f): f\in S(V)\}\\
	&=\phi(S(V)).
	\end{align*} 
	Thus $S(V)$ is convex. Now by Proposition \ref{osi-aou}, $V$ is an approximate order unit space.
\end{proof}

{\bf Acknowledgements.} {The first author is thankful to the Department of Atomic Energy, Government of India for providing financial support.}

\bibliographystyle{amsplain}

\end{document}